\documentclass{amsart}

\newtheorem{theorem}{Theorem}[section]
\newtheorem{lemma}[theorem]{Lemma}
\newtheorem{proposition}[theorem]{Proposition}

\theoremstyle{definition}

\theoremstyle{remark}
\newtheorem{remark}[theorem]{Remark}

\numberwithin{equation}{section}

\usepackage{pict2e}
\usepackage{longtable}
 \def\Q{\mathbb Q}  
\def\P{\mathbb P} 

\def\res{$58$ }
\def\resld{$27$ }
\def\resle{$31$ }
\def\tabld{$56$ }
\def\resfour{$57$ }
\def\tab{$14$ }

\begin{document}

\title{Gorenstein $\mathbb{Q}$-homology projective planes}

\author{DongSeon Hwang}
\address{Department of Mathematics, Ajou University, Suwon 443-749, Korea}
\email{dshwang@ajou.ac.kr}
%\thanks{}

%\author{Author Two}
\author{JongHae Keum}
\address{School of Mathematics, Korea Institute For Advanced Study, Seoul 130-722, Korea}
\email{jhkeum@kias.re.kr}

\author{Hisanori Ohashi}
\address{Department of Mathematics, Faculty of Science and Technology, Tokyo University of Science, 2641 Yamazaki, Noda, Chiba 278-8510, JAPAN}
\email{ohashi.hisanori@gmail.com}


%    General info

\date{December 27, 2013 and, November 19, 2014 in revised form.}


\begin{abstract}
We present the complete list of all singularity types on
Gorenstein $\mathbb{Q}$-homology projective planes, i.e.,
normal projective surfaces of second Betti number one with at worst rational double points. The list consists of 58 possible singularity types, each except two types supported by an example. 
\end{abstract}

\maketitle

\section{Introduction}
A normal projective surface with quotient singularities whose second Betti number is 1 is called a \textit{$\mathbb{Q}$-homology projective plane}. It is easy to see by considering the Albanese map that such a surface has first Betti number $0$, hence has the same Betti numbers with the complex projective plane. 
  A $\mathbb{Q}$-homology projective plane has Picard number $1$ and either $\pm K$ is ample or $K$ is numerically trivial. The minimal resolution of a $\mathbb{Q}$-homology projective plane has $p_g = q = 0$. A $\mathbb{Q}$-homology projective plane is said to be \textit{Gorenstein} if its singularities are all rational double points. A combination of rational double points will be denoted by its singularity type, a combination of $A_n$, $D_n$ and $E_n$, e.g., the singularity type $2A_3\oplus 3A_1$ denotes 2 rational double points of type $A_3$ and 3 of type $A_1$.

Gorenstein $\mathbb{Q}$-homology projective planes with ample anti-canonical divisor are   called  \textit{Gorenstein log del Pezzo surfaces of rank 1} and classified by Furushima \cite{Furushima}, Miyanishi and Zhang \cite{MZ}, and by Ye \cite{Ye}: there are exactly 27 singularity types on such surfaces and the surfaces with each singularity type were classified. The list of these 27 types will be denoted by ${\bf L}_{K<0}$.

Let ${\bf L}$ be the list of all singularity types on Gorenstein $\mathbb{Q}$-homology projective planes, and ${\bf L}_{K \equiv 0}$ (resp. ${\bf L}_{K\not\equiv 0}$) be the sub-list corresponding to the case where the canonical class $K$ is numerically trivial (resp. not trivial).  Then ${\bf L}_{K\not\equiv 0}$ contains two sub-lists ${\bf L}_{K<0}$ and ${\bf L}_{K>0}$, and the two sub-lists  may overlap each other.  Since a Gorenstein $\mathbb{Q}$-homology projective plane has rational double points only, it can be shown (see Section 2) that the rank of the singularity type is at most 9 and is equal to 9 if and only if $K\equiv 0$. Thus ${\bf L}$ is the disjoint union of ${\bf L}_{K\not\equiv 0}$ and ${\bf L}_{K\equiv 0}$. Not much has been known about
 ${\bf L}$  except the sub-list ${\bf L}_{K<0}$. In this note, we obtain the list ${\bf L}$.

\begin{theorem}\label{main} The list ${\bf L}$ consists of  \res singularity types, each except the two types $2A_3 \oplus A_2\oplus A_1$ and $A_3 \oplus 3A_2$ supported by an example $($Subsection \ref{real} and Table \ref{Kondo}$)$.

\medskip
 ${\bf L}_{K\not\equiv 0}= {\bf L}_{K<0}$ $($\resld types$):$
$A_8$, $A_7 \oplus A_1$, $A_5 \oplus A_2 \oplus A_1$, $2A_4$, $2A_3 \oplus 2A_1$, $4A_2$, $E_8$, $E_7 \oplus A_1$, $E_6 \oplus A_2$, $D_8$, $D_6 \oplus 2A_1$, $D_5 \oplus A_3$, $2D_4$, $A_7$, $A_5 \oplus A_2$, $2A_3 \oplus A_1$, $E_7$, $D_6 \oplus A_1$, $D_4 \oplus 3A_1$, $A_5 \oplus A_1$, $3A_2$, $E_6$, $A_3 \oplus 2A_1$, $D_5$, $A_4$, $A_2 \oplus A_1$, $A_1$, 

\medskip
${\bf L}_{K \equiv 0}$ $($\resle types$):$
$A_9$, $A_8 \oplus A_1$, $A_7 \oplus A_2$, $A_7 \oplus 2A_1$,  $A_6 \oplus A_2 \oplus A_1$,   $A_5 \oplus A_4$, $A_5 \oplus A_3 \oplus A_1$, $A_5 \oplus 2A_2$, $A_5 \oplus A_2 \oplus 2A_1$, $2A_4 \oplus A_1$, $A_4 \oplus A_3 \oplus 2A_1$, $3A_3$, $2A_3 \oplus 3A_1$,
$D_9$, $D_8 \oplus A_1$, $D_7 \oplus 2A_1$, $D_6 \oplus A_3$, $D_6 \oplus A_2 \oplus A_1$, $D_6 \oplus 3A_1$, $D_5 \oplus A_4$, $D_5 \oplus A_3 \oplus A_1$, $D_5 \oplus D_4$,  $D_4 \oplus A_3 \oplus 2A_1$, $2D_4 \oplus A_1$,
$E_8 \oplus A_1$, $E_7 \oplus A_2$, $E_7 \oplus 2A_1$, $E_6 \oplus A_3$, $E_6 \oplus A_2 \oplus A_1;$ $2A_3 \oplus A_2\oplus A_1$, $A_3 \oplus 3A_2$.

 We do not know whether any of the last two types realizes or not. 
\end{theorem}

The proof goes as follows.

When $K\not\equiv 0$, we only use arithmetic argument and lattice theory, without geometric argument, to obtain a list of 27 types, hence must coincide with the list ${\bf L}_{K<0}$ of \cite{MZ} and \cite{Ye}.  
Our method is different from theirs which used geometric structure of log del Pezzo surfaces.

When $K\equiv 0$, the rank of the singularity type is 9 and  we divide into two cases: the number of singular points  is at least 5; at most 4. In the first case there is only one type, namely $2A_3 \oplus 3A_1$, by the result of Hwang and Keum (\cite{HK1}, Theorem 1.1), whose proof requires not only arithmetic argument but also certain geometric argument, e.g., the possibility $D_5 \oplus 4A_1$ was ruled out by geometric argument.  In the second case we only use arithmetic argument to obtain a list of 30 types, which together with the type $2A_3 \oplus 3A_1$ yields the desired list ${\bf L}_{K \equiv 0}$.  Since $K\equiv 0$ and the surface has rational double points only, the minimal resolution of the surface has $p_g=q=0$ and $K\equiv 0$, hence is an Enriques surface. A singularity type gives on the minimal resolution a configuration of the same type of nine smooth rational curves. Each of the \resle types except the two types can be realized by contracting a  configuration of nine smooth rational curves on a suitable Enriques surface. For the realization, see Subsection \ref{real}, where we use the Enriques surfaces constructed by Kond\={o} \cite{Kondo}.   We do not know how many Enriques surfaces with given configuration type exist in general. There are only finitely many Enriques surfaces with a configuration of type $2A_3 \oplus 3A_1$ \cite{Keum2}. Section \ref{example} provides the construction of one-dimensional families of Enriques surfaces with a configuration of type $D_8 \oplus A_1$ or $E_7 \oplus 2A_1$.

\begin{remark} Consider the list ${\bf L}_{K>0}$ corresponding to the case where $K$ is ample. Among the 27 possible singularity types in ${\bf L}_{K\not\equiv 0}$ the four types  $A_1$, $A_2 \oplus A_1$,  $A_3 \oplus 2A_1$ and $D_4 \oplus 3A_1$ can be ruled out by the orbifold version of Bogomolov-Miyaoka-Yau inequality (e.g. \cite{Miyaoka}). Thus the list ${\bf L}_{K>0}$ consists of at most 23 types.  Since $K$ is ample, the minimal resolution is a minimal surface of general type with $p_g=q=0$.
Examples were found  by Keum \cite{Keum} for the singularity types $3A_2$ and $4A_2$, and by Frapporti \cite{Frapporti} for $2A_3 \oplus 2A_1$ and $2A_3 \oplus A_1$. For the remaining $19$ types the existence is not known.
\end{remark}

\medskip
 Throughout this paper, we work over the field $\mathbb{C}$ of complex numbers.

\section{Proof of Theorem \ref{main}}
\subsection{${\bf L}$ consists of at most \res types.}
In this subsection, we prove the following.

\begin{proposition}\label{cl} Let $\bar{S}$ be a Gorenstein $\mathbb{Q}$-homology projective plane. Then the singularity type of $\bar{S}$ is one of the \res types listed in Theorem \ref{main}.
\end{proposition}

If $|Sing(\bar{S})| \geq 5$, then the singularity type of $\bar{S}$ must be $2A_3 \oplus 3A_1$ by the following result.

\begin{theorem}$($\cite[Theorem 1.1]{HK1}, \cite[Theorem 1.2]{Be}$)$\label{hk}
Let $\bar{S}$ be a $\mathbb{Q}$-homology projective plane (with quotient singularities). Then $\bar{S}$ contains at most $5$ singular points, and if $\bar{S}$ contains $5$ singular points, then $\bar{S}$ is a Gorenstein log Enriques surface with singularity type $2A_3 \oplus 3A_1$.
\end{theorem}

Now assume that
\begin{equation}\label{leq4} |Sing(\bar{S})| \leq 4.
\end{equation}
 Since $\bar{S}$ has Picard number one, either $\pm K_{\bar{S}}$ is ample or $K_{\bar{S}}$ is numerically trivial, hence $K^2_{\bar{S}} \geq 0$ with equality iff $K_{\bar{S}}\equiv 0$. Let $$f:S \rightarrow \bar{S}$$
be the minimal resolution and $$R \subset H^2(S,  \mathbb{Z})_{free} := H^2(S,  \mathbb{Z})/(\text{torsion}) $$  be the sublattice  spanned by the numerical classes of the exceptional curves of $f:S \rightarrow \bar{S}$.  Then $S$ has second Betti number $1+{\rm rank}(R)$, $p_g=q=0$, thus $K_S^2 = 9-{\rm rank}(R)$ by Noether's formula. Since $\bar{S}$ has rational double points only, $K_S=f^*K_{\bar{S}}$  and $$K^2_{\bar{S}}=K_S^2=9-{\rm rank}(R)\geq 0,$$ hence
\begin{equation}\label{leq9} {\rm rank}(R) \leq 9
\end{equation}
with equality iff $K_{\bar{S}}\equiv 0$. The lattice $R$ and its singularity type will be denoted by the same notation, a direct sum of root lattices $A_n$, $D_n$ and $E_n$. It is easy to write down all possible singularity types satisfying the conditions (\ref{leq4}) and (\ref{leq9}). There are $127$ possible types:

\medskip

\begin{itemize}
\item the \resfour types in Theorem \ref{main} (except the type $2A_3 \oplus 3A_1$),
\item the \tabld types in Table \ref{det-table},
\item the \tab  types in Lemma \ref{lattice}.
\end{itemize}

\medskip

The \tabld types in Table \ref{det-table} are excluded by the following useful lemma.

\begin{lemma}\label{det}$($\cite[Lemma 3.3]{HK1}$)$
Let $\bar{S}$ be a $\mathbb{Q}$-homology projective plane (with quotient singularities). If $K_{\bar{S}}\not\equiv 0$, then
$|det(R)| \cdot K^2_{\bar{S}}$ is a square number.
\end{lemma}

\begin{table}[ht]
\caption{The 56 types excluded by Lemma \ref{det}}
\label{det-table}
$\begin{array}{||c|c||c|c||}
\hline
\textrm{Singularity Type} & |det(R)| \cdot K^2_{\bar{S}} & \textrm{Singularity Type} & |det(R)| \cdot K^2_{\bar{S}} \\
\hline
A_6 \oplus A_2 & 3 \cdot 7   &      A_2 \oplus 3A_1 & 2^5 \cdot 3 \\
A_6 \oplus 2A_1 &  2^2 \cdot 7   & A_3 \oplus A_1  & 2^3 \cdot 5  \\
A_5 \oplus A_3 & 2^3  \cdot 3   &  2A_2 & 3^2 \cdot 5 \\
A_5 \oplus 3A_1 & 2^4  \cdot 3   & A_2 \oplus 2A_1  & 2^2 \cdot  3 \cdot 5  \\
A_4 \oplus A_3 \oplus A_1 & 2^3  \cdot 5   & 4A_1 & 2^4 \cdot 5   \\
A_4 \oplus 2A_2 & 3^2  \cdot 5   & A_3 & 2^3 \cdot 3 \\
A_4 \oplus A_2 \oplus 2A_1 & 2^2  \cdot 3 \cdot 5   & 3A_1 & 2^4 \cdot 3 \\
2A_3 \oplus A_2 & 2^4  \cdot 3   &  A_2 & 3 \cdot 7\\
A_3 \oplus 2A_2 \oplus A_1 & 2^3  \cdot 3^2   & 2A_1 & 2^2 \cdot 7\\
A_6 \oplus A_1 & 2^2 \cdot 7   &            E_6 \oplus 2A_1 & 2^2 \cdot 3 \\
A_5 \oplus 2A_1 & 2^4 \cdot 3   &          D_7 \oplus A_1 & 2^3 \\
A_4 \oplus A_3 & 2^3 \cdot 5   &           D_6 \oplus A_2 & 2^2 \cdot 3\\
A_4 \oplus A_2 \oplus A_1   & 2^2 \cdot 3 \cdot 5             &  D_5 \oplus A_2 \oplus A_1 & 2^3 \cdot 3\\
A_4 \oplus 3A_1 & 2^4 \cdot 5  & D_5 \oplus 3A_1 & 2^5  \\
A_3 \oplus 2A_2 & 2^3 \cdot 3^2   &    D_4 \oplus A_4 & 2^2 \cdot 5 \\
A_3 \oplus A_2 \oplus 2A_1 & 2^5 \cdot 3   & D_4 \oplus A_3 \oplus A_1 & 2^5 \\
3A_2 \oplus A_1 & 2^2 \cdot 3^3   & D_4 \oplus A_2 \oplus 2A_1 & 2^4 \cdot 3\\
A_6 & 3 \cdot 7 &  E_6 \oplus A_1 & 2^2 \cdot 3\\
A_4 \oplus A_2 &      3^2 \cdot 5  &                        D_7 & 2^3\\
A_4 \oplus 2A_1  & 2^2 \cdot  3 \cdot 5    &               D_5 \oplus  A_2 & 2^3 \cdot 3\\
2A_3  & 2^4 \cdot  3   & D_5 \oplus  2A_1 & 2^5   \\
A_3 \oplus A_2 \oplus A_1 & 2^3 \cdot 3^2 &             D_4 \oplus  A_3 & 2^5    \\
A_3 \oplus 3A_1 & 2^5 \cdot 3 & D_4 \oplus A_2 \oplus A_1 & 2^4 \cdot 3\\
2A_2 \oplus 2A_1 & 2^2 \cdot 3^3   & D_6 & 2^2 \cdot 3\\
A_5 & 2^3 \cdot 3    & D_5 \oplus  A_1 & 2^3 \cdot 3   \\
A_4 \oplus A_1 & 2^3 \cdot 5  & D_4 \oplus  2A_1 & 2^4 \cdot 3   \\
A_3 \oplus A_2  & 2^4 \cdot 3   & D_4 \oplus  A_1 & 2^5 \\
2A_2 \oplus A_1 & 2^3 \cdot 3^2    & D_4 & 2^2 \cdot 5\\
\hline
\end{array}$
\end{table}

It remains to remove the \tab  types in Lemma \ref{lattice}. The cohomology lattice  $H^2(S,  \mathbb{Z})_{free}$ is unimodular by Poincar\'e duality, and of signature $(1,{\rm rank}(R))$.  Any indefinite unimodular lattice is up to isometry determined by its signature and parity (i.e., whether it is odd or even) (cf. \cite{BHPV}, p. 18). 

Suppose that ${\rm rank}(R) = 9$. Then $K_{\bar{S}}\equiv 0$ and the minimal resolution $S$ is an Enriques surface. The cohomology lattice  $H^2(S,  \mathbb{Z})_{free}$ of an Enriques surface is an even unimodular lattice $II_{1,9}$ of signature $(1, 9)$. Thus
$$R \subset  H^2(S,  \mathbb{Z})_{free}  \cong II_{1,9}\cong H \oplus E_8$$
where $H$ is the even unimodular lattice of signature $(1, 1)$, and $E_8$ is the even unimodular lattice of signature $(0, 8)$.

Suppose that ${\rm rank}(R)<9$. Then the minimal resolution $S$ is either a rational surface or a minimal surface of general type with $p_g = 0$. The cohomology lattice $H^2(S,  \mathbb{Z})_{free}$ is odd or even. If it is even, then it must be isometric to the hyperbolic lattice $H$ of signature $(1,1)$, which occurs when $S$ is a minimal rational ruled surface with a section of self intersection $-2$.  Thus
$$R \subset  H^2(S,  \mathbb{Z})_{free}   \cong I_{1, {\rm rank}(R)} \,\,{\rm or}\,\, H,$$
where $I_{1,m}$ denotes an odd unimodular lattice of signature $(1,m)$.

The following lemma completes the proof of Proposition \ref{cl}. 

\begin{lemma}\label{lattice}
\begin{enumerate}
\item Let $R = A_6 \oplus A_3$, $A_6 \oplus 3A_1,$ $A_4 \oplus A_3 \oplus A_2$, $A_4 \oplus 2A_2 \oplus A_1$, $D_7 \oplus A_2$, $D_5 \oplus 2A_2$, $D_5 \oplus A_2 \oplus 2A_1$, $D_4 \oplus A_5$, $D_4 \oplus A_4 \oplus A_1$, $D_4 \oplus A_3 \oplus A_2$, $D_4 \oplus 2A_2 \oplus A_1$, or $E_6 \oplus 3A_1$. Then the lattice $R$ cannot be embedded in the lattice $H \oplus E_8$.
\item Let $R = D_4 \oplus A_2$ or $D_4 \oplus 2A_2$. Then the lattice $R$ cannot be embedded in the lattice $I_{1, L}$ where $L = rank(R)$.
\end{enumerate}
\end{lemma}

\begin{proof}
(1) Suppose that $R$ is embedded in $H \oplus E_8$. Let $R^{\perp}$ be the orthogonal complement of $R$ in $H \oplus E_8$. Then $(R \oplus R^{\perp}) \otimes \mathbb{Q}_p \cong (H \oplus E_8) \otimes \mathbb{Q}_p$ for any prime number $p$. Thus, by \cite[Theorem IV-9]{Serre}, their discriminants must be the same. It follows that $d_p(R)d_p(R^{\perp})  =d_p(R \oplus R^{\perp}) = d_ p(H \oplus E_8) = -1$ for every prime $p$. Therefore $(d_p(R), d_p(R^{\perp})) = 1$ for every prime $p$. Again by   \cite[Theorem IV-9]{Serre}, their epsilon invariants must be the same, i.e.,
  $\epsilon_p(R \oplus R^{\perp}) = \epsilon_p(H \oplus E_8)$ for every prime $p$. Since $\epsilon_p(R^{\perp}) = 1$, $\epsilon_p(R \oplus R^{\perp})= \epsilon_p(R) \epsilon_p( R^{\perp}) (d_p(R), d_p(R^{\perp})) = \epsilon_p(R).$ 
Since $\epsilon_p(H \oplus E_8) = 1$ for any prime $p$,  it is enough to show that $\epsilon_p(R) = -1$ for some prime $p$. 
For the definition and basic properties of epsilon invariant, see \cite[Section 6]{HK1}. 

\medskip
1) The case $A_6 \oplus A_3$\\
In this case, $$\epsilon_7(R) = \epsilon_7(A_6)\epsilon_7(A_3)(d(A_6), d(A_3))_7=1 \cdot 1 \cdot (7, -1)_7 = -1.$$

2) The case $A_6 \oplus 3A_1$\\
Since  $\epsilon_7(A_6 \oplus A_1) = \epsilon_7(A_6)\epsilon_7(A_1)  (d(A_6), d(A_1))_7 = 1 \cdot 1 \cdot (7,-2)_7 = -1$ and $\epsilon_7(2A_1) = (d(A_1), d(A_1))_7 = 1$, we have
$$\epsilon_7(R) = \epsilon_7(A_6 \oplus A_1)\epsilon_7(2A_1)(d(A_6 \oplus A_1), d(2A_1))_7 = (-1) \cdot 1 \cdot (-14, 1)_7= -1.$$

3) The case $A_4 \oplus A_3 \oplus A_2$\\
Since $\epsilon_5(A_4 \oplus A_3) = \epsilon_5(A_4)\epsilon_5(A_3)(d(A_4), d(A_3))_5 = 1 \cdot 1 \cdot (5, -1)_5 = 1$, we have
$$\epsilon_5(R) = \epsilon_5(A_4 \oplus A_3) \epsilon_5(A_2) (d(A_4 \oplus A_3), d(A_2))_5 = 1 \cdot 1 \cdot (-5, 3)_5 = -1.$$
 
4) The case $A_4 \oplus 2A_2 \oplus A_1$\\
Since $\epsilon_3(A_4 \oplus A_1) = \epsilon_3(A_4)\epsilon_3(A_1)(d(A_4), d(A_1))_3 = 1 \cdot 1 \cdot (5, -2)_3 = 1$ and  $\epsilon_3(2A_2) = (d(A_2), d(A_2))_3 = -1$, we have 
$$\epsilon_3(R) = \epsilon_3(A_4 \oplus A_1)\epsilon_3(2A_2) (d(2A_2), d(A_4 \oplus A_1))_3 = 1 \cdot (-1) \cdot 1 = -1.$$

5) The case $D_7 \oplus A_2$\\
In this case,   $$\epsilon_3(R) =  \epsilon_3(D_7) \epsilon_3(A_2)  (d(D_7),  d(A_2))_3 = 1 \cdot 1 \cdot (-1, 3)_3 = -1.$$  

6) The case $D_5 \oplus 2A_2$\\
In this case, 
$$\epsilon_3(R) =  \epsilon_3(D_5) \epsilon_3(2A_2)  (d(D_5),  d(2A_2))_3 = 1 \cdot (-1) \cdot (-1,1)_3 = -1.$$  

7) The case $D_5 \oplus A_2 \oplus 2A_1$\\
Since $\epsilon_3(D_5 \oplus A_2) = \epsilon_3(D_5)\epsilon_3(A_2)(d(D_5), d(A_2))_3 = 1 \cdot 1 \cdot (-1, 3)_3 = -1$ and  $\epsilon_3(2A_1) = (d(A_1), d(A_1))_3 = (-2, -2)_3 = 1$, we have
$$\epsilon_3(R) = \epsilon_3(D_5 \oplus A_2)\epsilon_3(2A_1)  (d(D_5 \oplus A_2), d(2A_1))_3 = (-1) \cdot 1 \cdot 1 = -1.$$

8) The case $D_4 \oplus A_5$\\
In this case, $$\epsilon_3(R) =  \epsilon_3(D_4) \epsilon_3(A_5)  (d(D_4),  d(A_5))_3 = 1 \cdot (-1) \cdot (1, -6)_3 = -1.$$  

9) The case $D_4 \oplus A_4 \oplus A_1$\\
Since $\epsilon_5(D_4 \oplus A_4) = \epsilon_5(D_4) \epsilon_5(A_4) (d(D_4), d(A_4))_5 = 1 \cdot 1 \cdot (1, 5)_5 = 1$, we have 
$$ \epsilon_5(R) = \epsilon_5(D_4 \oplus A_4)\epsilon_5(A_1) (d(D_4 \oplus A_4), d(A_1))_5 = 1 \cdot 1 \cdot (5, -2)_5 = -1.$$

10) The case  $D_4 \oplus A_3 \oplus A_2$\\
Since $\epsilon_3(D_4 \oplus A_2) = \epsilon_3(D_4)\epsilon_3(A_2) (d(D_4), d(A_2))_3 = 1 \cdot 1 \cdot(1, -3)_3 = 1$, we have 
$$ \epsilon_3(R)  = \epsilon_3(D_4 \oplus A_2) \epsilon_3(A_3) (d(D_4 \oplus A_2), d(A_3))_3 = 1 \cdot 1 \cdot (3, -1)_3 = -1.$$

11) The case   $D_4 \oplus 2A_2 \oplus A_1$\\
Since $\epsilon_3(D_4 \oplus A_1) = \epsilon_3(D_4)\epsilon_3(A_1) (d(D_4), d(A_1))_3 = 1 \cdot 1 \cdot 1 = 1$ and  $\epsilon_3(2A_2) = -1$, we have
$$\epsilon_3(R) = \epsilon_3(D_4 \oplus A_1) \epsilon_3(2A_2) (d(D_4 \oplus A_1), d(2A_2))_3 = 1 \cdot (-1) \cdot 1= -1.$$ 

12) The case  $E_6 \oplus 3A_1$\\
Since $\epsilon_3(E_6 \oplus A_1) = \epsilon_3(E_6) \epsilon_3(A_1) (d(E_6), d(A_1))_3 = (-1) \cdot 1 \cdot (3, -2)_3 = -1$ and $\epsilon_3(2A_1) = (d(A_1), d(A_1))_3 = 1$, we have 
$$  \epsilon_3(R) = \epsilon_3(E_6 \oplus A_1)\epsilon_3(2A_1)(d(E_6 \oplus A_1), d(2A_1))_3 = (-1) \cdot 1 \cdot (-6, 1)_3= -1.$$

\medskip
(2) Suppose that $R$ is embedded in $I_{1, L}$. Let $R^{\perp}$ be the orthogonal complement of $R$ in $I_{1, L}$.  Then $(R \oplus R^{\perp}) \otimes \mathbb{Q}_p \cong I_{1, L} \otimes \mathbb{Q}_p$ for each prime $p$. Again by  \cite[Theorem IV-9]{Serre},
  $\epsilon_p(I_{1, L}) = \epsilon_p(R \oplus R^{\perp})$ for any prime $p$.
Since $\epsilon_3(I_{1, L}) = 1$ for any positive integer $L$, it is enough to show that $\epsilon_3(R \oplus R^{\perp}) = -1$.

\medskip
1) The case $D_4 \oplus A_2$\\
Since $\epsilon_3(R) = \epsilon_3(D_4 \oplus A_2) =1$ by the computation in (1), we have
$$\epsilon_3(R \oplus R^{\perp}) = \epsilon_3(R)\epsilon_3(R^{\perp})(d(R), d(R^{\perp}))_3 =
1 \cdot 1 \cdot (3, 3)_3 = -1.$$

2) The case $D_4 \oplus 2A_2$ \\
Since $\epsilon_3(R) = \epsilon_3(D_4 \oplus A_2) \epsilon_3(A_2) (d(D_4\oplus A_2), d(A_2))_3 = 1 \cdot 1 \cdot (3,3) = -1$, we have
$$ \epsilon_3(R \oplus R^{\perp}) = \epsilon_3(R)\epsilon_3(R^{\perp})(d(R), d(R^{\perp}))_3 =
(-1) \cdot 1 \cdot (1, 1)_3 = -1.$$
\end{proof}

\subsection{Supporting Examples.}\label{real}
We complete the proof of Theorem \ref{main} by providing a supporting example for each of the 56 singularity types. 

If $R$ is one of the  \resld singularity types with ${\rm rank}(R)<9$, then there exists a Gorenstein log del Pezzo surface of rank 1 with the given  singularity type $R$ by \cite[Theorem 1.2]{Ye}.

If $R$ is one of the types in  ${\bf L}_{K\equiv 0}$ except the two types  $2A_3 \oplus A_2 \oplus A_1$ and $A_3 \oplus 3A_2$, then it  is supported by an example obtained by contracting suitable nine smooth rational curves on an Enriques surface with finite automorphism group constructed by Kond\={o}(\cite{Kondo}). These examples are listed in  Table \ref{Kondo}, where  we use the same notation for smooth rational curves as in his paper. In Section 3  more impressive examples will be obtained, in a different way, from Enriques surfaces with infinite automorphism group.   

\begin{table}[ht]
\caption{Realization of the 29 singularity types in ${\bf L}_{K\equiv 0}$}\label{Kondo}
$\begin{array}{|c|c|c|}
\hline
\textrm{Singularity} & \textrm{Kond\={o}'s} & \textrm{Nine  smooth rational curves}\\
\textrm{Type} & \textrm{Example} & \\
\hline
A_9 & \textrm{Type VII} & (E_1, E_2, E_3, E_4, E_5, E_6, E_7, E_8, E_{12})  \\
A_8 \oplus A_1 & \textrm{Type I} &  (F_1, F_2, F_3, F_4, F_5, F_6, F_7, F_{10}); (F_{11}) \\
A_7 \oplus A_2 & \textrm{Type V} & (E_1, E_2, E_3, E_4, E_5, E_7, E_8); (E_{14}, E_{15})  \\
A_7 \oplus 2A_1 & \textrm{Type I} & (F_1, F_2, F_3, F_4, F_5, F_6, F_8); (F_{10}); (F_{11}) \\
A_6 \oplus A_2 \oplus A_1 & \textrm{Type V}  & (E_1, E_2, E_3, E_4, E_5, E_7); (E_{14}, E_{15}); (E_{11})   \\
A_5 \oplus A_4 &  \textrm{Type VI} & (E_1, E_2, E_3, E_4, E_5); (E_{11}, E_{14}, E_{15}, E_{19}) \\
A_5 \oplus A_3 \oplus A_1 & \textrm{Type VI}  & (E_1, E_2, E_3, E_4, E_5); (E_{14}, E_{15}, E_{19}); (E_8)\\
A_5 \oplus 2A_2 &  \textrm{Type V} & (E_1, E_2, E_3, E_4, E_6); (E_7, E_8); (E_{14}, E_{15}) \\
A_5 \oplus A_2 \oplus 2A_1 &  \textrm{Type V} & (E_1, E_2, E_3, E_4, E_6); (E_{14}, E_{15}); (E_7); (E_{11}) \\
2A_4 \oplus A_1 & \textrm{Type VI}  & (E_1, E_2, E_3, E_4); (E_{13}, E_{14}, E_{15}, E_{19}); (E_8) \\
A_4 \oplus A_3 \oplus 2A_1 & \textrm{Type III}  & (E_1, E_2, E_3, E_4); (E_6, E_7, E_{12}); (E_{13}); (E_{14}) \\
3A_3 & \textrm{Type II} & (F_1, F_2, F_4); (F_5, F_6, F_8); (F_9, F_{10}, F_{12}) \\
2A_3 \oplus 3A_1 & \textrm{Type III}   & (E_1, E_2, E_9); (E_6, E_7, E_{12}); (E_4); (E_{13}); (E_{14})   \\
D_9 & \textrm{Type II}  & (F_1, F_2, F_3, F_5, F_6, F_7, F_9, F_{10}, F_{12}) \\
D_8 \oplus A_1 & \textrm{Type I} & (F_1, F_2, F_3, F_4, F_6, F_7, F_8, F_9); (F_{12}) \\
D_7 \oplus 2A_1 & \textrm{Type III}  & (E_1, E_2, E_3, E_4, E_5, E_6, E_{12});  (E_{13}); (E_{14}) \\
D_6 \oplus A_3 & \textrm{Type II}  & (F_1, F_2, F_3, F_5, F_6, F_8); (F_9, F_{10}, F_{12})  \\
D_6 \oplus A_2 \oplus A_1 & \textrm{Type V}  & (E_1, E_2, E_3, E_4, E_6, E_9); (E_{15}, E_{19}); (E_{11}) \\
D_6 \oplus 3A_1 & \textrm{Type V}  & (E_1, E_2, E_3, E_4, E_6, E_9); (E_7); (E_{11}); (E_{15}) \\
D_5 \oplus A_4 &  \textrm{Type IV}  &  (E_1, E_2, E_3, E_{16}, E_{20}); (E_5, E_6, E_9, E_{17}) \\
D_5 \oplus A_3 \oplus A_1 & \textrm{Type IV}  & (E_5, E_6, E_7, E_9, E_{17});  (E_1, E_2, E_3); (E_{20}) \\
D_5 \oplus D_4 & \textrm{Type IV}  & (E_5, E_6, E_9, E_{17}, E_{18}); (E_1, E_2, E_3, E_{16}) \\
D_4 \oplus A_3 \oplus 2A_1 & \textrm{Type III}  & (E_4, E_5, E_6, E_{12}); (E_1, E_2, E_9);  (E_{13}); (E_{14}) \\
2D_4 \oplus A_1 & \textrm{Type III}  & (E_1, E_2, E_8, E_9); (E_4, E_5, E_6, E_{12}); (E_{14})  \\
E_8 \oplus A_1 & \textrm{Type I} & (F_1, F_2, F_3, F_4, F_5, F_6, F_7, F_9); (F_{12})   \\
E_7 \oplus A_2 & \textrm{Type V}   & (E_1, E_2, E_3, E_4, E_5, E_7, E_9); (E_{15}, E_{17}) \\
E_7 \oplus 2A_1 & \textrm{Type V}  & (E_1, E_2, E_3, E_4, E_5, E_7, E_9); (E_{11}); (E_{15})   \\
E_6 \oplus A_3 & \textrm{Type VI}  & (E_1, E_2, E_3, E_4, E_5, E_9); (E_{11}, E_{16}, E_{19}) \\
E_6 \oplus A_2 \oplus A_1 & \textrm{Type V}  & (E_1, E_2, E_3, E_4, E_5, E_9); (E_{15}, E_{19}); (E_{11})  \\
\hline
\end{array}$
\end{table}

\begin{remark}
Neither the type $2A_3 \oplus A_2 \oplus A_1$ nor $A_3 \oplus 3A_2$ can be obtained from  Kond\={o}'s examples of Enriques surfaces.  In other words, the 29 types in Table \ref{Kondo} are the only types that can be obtained from  Kond\={o}'s examples.
This can be checked by using computer algebra system, e.g., Maple. The main steps for the algorithm are as follows: for each of the 8 types of Kond\={o}'s examples,
\begin{enumerate}
\item generate the dual graph $G$ of the intersection matrix of all $(-2)$-curves on the surface (an Enriques surface with finite automorphism group has finitely many $(-2)$- curves),
\item for each set $V$ of 9 vertices of $G$, check if the subgraph $H$ of $G$ generated by $V$ is a disjoint union of graphs of Dynkin type $A_n$, $D_n$, or $E_n$. 
\item collect, up to graph isomorphism, all subgraphs $H$ that survived the step (2).
\end{enumerate}

If $2A_3 \oplus A_2 \oplus A_1$ or $A_3 \oplus 3A_2$ exists, then it must  realize on an Enriques surface with infinite automorphism group.
\end{remark}

\begin{remark} Both lattices $2A_3 \oplus A_2 \oplus A_1$ and $A_3 \oplus 3A_2$ can be embedded into the Enriques lattice $H\oplus E_8$. More precisely, $$2A_3 \oplus A_1 \oplus A_2\subset E_7\oplus A_2\subset H\oplus E_8,$$ $$A_3\oplus 3A_2\subset A_3\oplus E_6\subset  H\oplus E_8.$$ In fact, these are well-known.  To show $2A_3 \oplus A_1\subset E_7$, let $e_1, e_2, \ldots, e_7$ be  $(-2)$-vectors  generating the lattice $E_7$, whose Dynkin diagram is given by 

 $$
\begin{picture}(100,40)
\put(0,25){$e_1-e_2-e_3-e_4-e_5-e_6$}
 \put(50,15){\line(0,1){6}}
 \put(45,5){$e_7$}
\end{picture}
$$
Take $e_0= -2e_1-3e_2-4e_3-3e_4-2e_5-e_6-2e_7.$ Then $e_0\in E_7$, $e_0^2=-2$, and we have the following extended Dynkin diagram 

 $$
\begin{picture}(100,40)
\put(0,25){$e_0-e_1-e_2-e_3-e_4-e_5-e_6$}
 \put(70,15){\line(0,1){6}}
 \put(65,5){$e_7$}
\end{picture}
$$
which obviously contains the Dynkin diagram $2A_3 \oplus A_1$. To show $3A_2\subset E_6$, let $e_1, e_2, \ldots, e_6$ be  $(-2)$-vectors  generating the lattice $E_6$, with Dynkin diagram 

 $$
\begin{picture}(100,40)
\put(0,25){$e_1-e_2-e_3-e_4-e_5$}
 \put(50,15){\line(0,1){6}}
 \put(45,5){$e_6$}
\end{picture}
$$
Take $e_0= -e_1-2e_2-3e_3-2e_4-e_5-2e_6.$ Then $e_0\in E_6$, $e_0^2=-2$, and we have the following extended Dynkin diagram 

 $$
\begin{picture}(100,40)
\put(0,25){$e_1-e_2-e_3-e_4-e_5$}
 \put(50,15){\line(0,1){6}}
 \put(45,5){$e_6$}
 \put(50,-5){\line(0,1){6}}
 \put(45,-15){$e_0$}
\end{picture}
$$

\bigskip\bigskip \noindent
which obviously contains the Dynkin diagram $3A_2$.
\end{remark}

\section{Examples from Enriques surfaces with infinite automorphism group}\label{example}

In this section, we give further examples of Gorenstein $\Q$-homology projective planes
with numerically trivial canonical divisor. In this case, the minimal resolution  is an Enriques surface. In the previous section
we have already seen plenty of such examples arising from Enriques surfaces with finite automorphism group.
Here we construct new ones from Enriques surfaces with infinite automorphism group. We summarize the result as follows.

\begin{theorem}\label{ex-thm} Let $R$ be one of the two types $D_8 \oplus A_1$, $E_7 \oplus 2A_1$ in the list
${\bf L}_{K \equiv 0}$. 
Then there exists a one-dimensional family of Gorenstein $\Q$-homology projective planes with singularity type $R$ such that the minimal resolution of a very general member is an Enriques surface
with infinite automorphism group.
\end{theorem}

\begin{remark}
Any Gorenstein $\mathbb{Q}$-homology projective plane is an example of a Mori dream space, namely its Cox ring is finitely generated. Is the minimal resolution  of a $\mathbb{Q}$-homology projective plane also a Mori dream space? The finite generation of the Cox ring of $S$ implies the finiteness of the automorphism group $\mathrm{Aut}(S)$, so Theorem \ref{ex-thm} gives a negative answer to the question. The question is still open for rational surfaces (see \cite[Question 6.7]{HP}).
\end{remark}

Our Enriques surfaces arise from Kummer surfaces of special type.
It is known that a very general Kummer surface of product type $Km(E\times F)$
cover essentially two kinds of Enriques surfaces, namely the quotients by
{\em{Lieberman involutions}} and {\em{Kondo-Mukai involutions}} \cite{Ohashi07}.
The idea is to specialize them to the case of self-product type $Km(E\times E)$. \\

Now let the elliptic curve $E_{\lambda}$ be
\[E_{\lambda} \colon y^2=x(x-1)(x-\lambda ),\ (\lambda\neq 0,1)\]
defined in the Legendre form. To the abelian surface $A=E_{\lambda}\times E_{\lambda}$,
we can associate the {\em{Kummer surface of self-product type}}
$X=Km(A)$ as the minimal resolution of the quotient surface
$\overline{X}=A/(-1)_A$. Since the double covering
$E_{\lambda}\rightarrow E_{\lambda}/(-1)_{E_{\lambda}}\simeq \P^1$ branches at the
four points $x= 0,1,\lambda,\infty$, the natural map $\overline{X}\rightarrow
(E_{\lambda}/(-1)_{E_{\lambda}})\times (E_{\lambda}/(-1)_{E_{\lambda}})=\P^1\times \P^1=R$
is the double covering branched along the eight lines
\begin{equation}
\begin{split}
l_1\colon & x=0,\ l_2\colon x=1,\ l_3\colon x=\lambda,\ l_4\colon x=\infty, \\
l'_1\colon & x'=0,\ l'_2\colon x'=1,\ l'_3\colon x'=\lambda,\ l'_4\colon x'=\infty
\end{split}
\end{equation}
where $x$ and $x'$ are the inhomogeneous coordinates of two $\P^1$s.
This gives the another visible construction of $X$ as the desingularization of the
double covering of $R$.
We write $l_i\cap l'_j=\{p_{ij}\}$. See Figure \ref{24} and also \cite[Section 3]{Fields}.
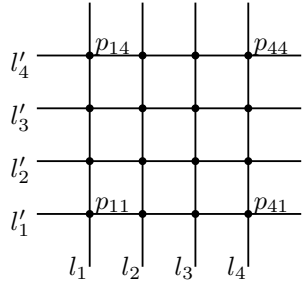
\begin{figure}[h]
\centering
\begin{picture}(100,100)
\linethickness{0.7pt}
\put(0,20){\line(1,0){100}}\put(-10,14){$l'_1$}
\put(0,40){\line(1,0){100}}\put(-10,34){$l'_2$}
\put(0,60){\line(1,0){100}}\put(-10,54){$l'_3$}
\put(0,80){\line(1,0){100}}\put(-10,74){$l'_4$}
\put(20,0){\line(0,1){100}}\put(12,-3){$l_1$}
\put(40,0){\line(0,1){100}}\put(32,-3){$l_2$}
\put(60,0){\line(0,1){100}}\put(52,-3){$l_3$}
\put(80,0){\line(0,1){100}}\put(72,-3){$l_4$}
%%%%%%%%
\put(20,20){\circle*{3}}
\put(20,40){\circle*{3}}
\put(20,60){\circle*{3}}
\put(20,80){\circle*{3}}
\put(40,20){\circle*{3}}
\put(40,40){\circle*{3}}
\put(40,60){\circle*{3}}
\put(40,80){\circle*{3}}
\put(60,20){\circle*{3}}
\put(60,40){\circle*{3}}
\put(60,60){\circle*{3}}
\put(60,80){\circle*{3}}
\put(80,20){\circle*{3}}
\put(80,40){\circle*{3}}
\put(80,60){\circle*{3}}
\put(80,80){\circle*{3}}
%%%%%%%
\put(22,22){$p_{11}$}
\put(82,82){$p_{44}$}
\put(82,22){$p_{41}$}
\put(22,82){$p_{14}$}
\end{picture}
\caption{Branches of $\overline{X}\rightarrow \P^1\times \P^1$}
\label{24}
\end{figure}
We denote by $L_i$ and $L_i'$ the smooth rational curves on $X$
dominating $l_i$ and $l'_i$ respectively. We denote by $E_{ij}$ the exceptional curve on $X$
which resolves the node on $\overline{X}$ corresponding to the point $p_{ij}$.
Thus we get $24$ smooth rational curves on $X$, which are often called
{\em{the double Kummer configuration}} of $X$.

By our choice of the abelian surface $A$, we can obtain further
rational curves on $X$.
Let us look at the four curves of bidegree $(1,1)$ on $R$ defined by the following equations.
\begin{equation}\label{11}
\begin{split}
c_1 &\colon x'=x\quad \text{(through $p_{11},p_{22},p_{33},p_{44}$.)}\\
c_2&\colon xx'=\lambda\quad \text{(through $p_{14},p_{23},p_{32},p_{41}$.)}\\
c_3&\colon (x-1)x'=x-\lambda\quad \text{(through $p_{13},p_{24},p_{31},p_{42}$.)}\\
c_4 &\colon (x-\lambda)x'=\lambda(x-1)\quad \text{(through $p_{12},p_{21},p_{34},p_{43}$.)}
\end{split}
\end{equation}
The special property owned by them is that each $c_k$ passes through four of the intersections $p_{ij}$ as indicated. Therefore, after discarding exceptional curves,
the inverse image of $c_k$ on $X$ decomposes into two smooth rational curves $C_k$ and $C_k'$.
We obtain eight additional smooth rational curves on $X$ in this way.
Their configuration is partially described as follows.
\begin{lemma}\label{splitting}
The curves $C_k$ and $C'_k$ are disjoint for $k=1,\dots,4$.
Let $k_1\neq k_2$. Then, the four curves $C_{k_1}, C'_{k_1}, C_{k_2}, C'_{k_2}$ form a
disjoint union of two Kodaira fibers of type $I_2$.
\end{lemma}
\begin{proof}
The first assertion is clear. We prove the second statement for $(k_1, k_2)=(1,2)$. The proof
for other pairs is similar.

The divisor $f=c_1+c_2$ has bidegree $(2,2)$ on $R$ and
passes through eight points $p_{11}, p_{14}, p_{22}, \dots, p_{44}$.
On the other hand, the divisors $f_1=l_1+l_4+l'_2+l'_3$ and $f_2=l_2+l_3+l'_1+l'_4$ also satisfy the
same conditions. Therefore these three divisors are in the elliptic pencil
\[\mathcal{L}=|\mathcal{O}_R(2,2)-p_{11}-p_{14}-p_{22}-p_{23}-p_{32}-p_{33}-p_{41}-p_{44}|.\]
It defines the rational elliptic surface $\hat{R}\rightarrow \mathbb{P}^1$. The surface $X$ is
nothing but (the resolution of) the double cover of $\hat{R}$ branched along the
two fibers $f_1$ and $f_2$. Thus, the fiber $f$ decomposes into two $I_2$ fibers on $X$
by this double cover, as stated.
\end{proof}

From now on we assume that $\lambda$ is very general.
\begin{proposition}\label{infinite_aut}
If $\lambda$ is very general, then the transcendental lattice of the Kummer surface $X$ has the Gram matrix
\[\begin{pmatrix}
0 & 2 & 0 \\ 2 & 0 & 0 \\ 0 & 0 & 4
\end{pmatrix}.\]
Therefore {\em{any}} Enriques surface covered by $X$ has an infinite automorphism group.
\end{proposition}
\begin{proof}
If $\lambda$ is very general, the abelian surface $A=E_{\lambda}\times E_{\lambda}$ has Picard number $3$. It is generated by the axes $E_{\lambda}\times \{0\}$ and
$\{0\}\times E_{\lambda}$ and the diagonal $\Delta =\{(a,a)\mid a\in E_{\lambda}\}$.
Their intersection numbers are given by the following Gram matrix
\[\begin{pmatrix}
0 & 1 & 1 \\ 1 & 0 & 1 \\ 1 & 1 & 0
\end{pmatrix}\]
and the Neron-Severi lattice of $A$ is isomorphic to $U\oplus A_1$
(where $U$ is the hyperbolic plane and $A_1=\langle -2\rangle$). Therefore the transcendental lattice of $A$
is isomorphic to $U\oplus \langle 2 \rangle$.
By passing to the Kummer surface, the transcendental lattice is multiplied by two.
Thus we get the result. The last sentence follows from the classification of
Kondo \cite{Kondo} of Enriques surfaces with finite automorphism groups.
\end{proof}
In what follows we give two kinds of Enriques quotients of $X$. \\

\noindent {\bf{The Lieberman involution}}.\quad
Let us give the first free involution $\varepsilon_1$ on $X$,
which gives the Enriques quotient $S_1$.
Recall that the points $p_{ij}$ are in one-to-one correspondence with $2$-torsion points of $A$. The zero element is at $p_{44}$. We consider the translation automorphism
$\tau\in \mathrm{Aut}(A)$ defined by the $2$-torsion point
corresponding to $p_{33}$. Then we can see that the composite
$\tau \circ (1_{E_{\lambda}},-1_{E_{\lambda}})\in\mathrm{Aut}(A)$
descends to $\overline{X}$ and to $X$ to give a fixed-point-free involution $\varepsilon_1$.
This is one of what is known as the {\em{Lieberman involution}} \cite{MN, Ohashi07}.
As the quotient, we obtain the Enriques surface $S_1=X/\varepsilon_1$.
In terms of the
double covering $\overline{X}\rightarrow R$, $\varepsilon_1$ is one of the two lifts of the small
involution
\[(x,x')\mapsto \left(\frac{\lambda(x-1)}{x-\lambda},\frac{\lambda(x'-1)}{x'-\lambda}\right). \]
To see how $\varepsilon_1$ permutes smooth rational curves on $X$,
we set $\sigma=(12)(34)\in \mathfrak{S}_4$. The following facts can be verified.
\begin{itemize}
\item The involution $\varepsilon_1$ permutes the curves $L_i,L_i', E_{ij}$
by the action of $\sigma$ on indices; $\varepsilon_1 (L_1)=L_2, \varepsilon_1
(E_{13})=E_{24}$ for example.
\item The involution $\varepsilon_1$ exchanges $C_k$ and $C'_k$ since $C_k\cup C_k'$ is preserved and $\varepsilon_1$ exchanges the two sheets of $\overline{X}/R$.
\end{itemize}
From these facts, we can draw the dual graph of smooth rational curves on the Enriques quotient $S_1=X/\varepsilon_1$.
We denote by $L_{12},L_{34},L'_{12},L'_{34}$ the rational curves on $S_1$ corresponding to
$L_1+L_2, L_3+L_4, L'_1+L'_2, L'_3+L_4'$ on $X$ respectively.
We denote by $F_{ij}$ the rational curve on $S_1$
corresponding to $E_{ij}+E_{\sigma(i)\sigma(j)}$ (our choice is such that $i<\sigma (i)$).
Also we denote by $D_k$ the curve on $S_1$ corresponding to $C_k+C_k'$.
The dual graph of the the set of curves $\{L_{ij},L'_{ij},F_{ij}\}$ is depicted in Figure \ref{12A}.

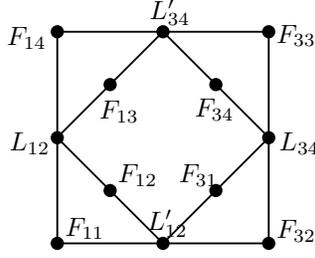
\begin{figure}[h]
\centering
\begin{picture}(80,80)
\linethickness{0.7pt}
\put(0,0){\line(1,0){80}}
\put(0,0){\line(0,1){80}}
\put(80,0){\line(0,1){80}}
\put(0,80){\line(1,0){80}}
\put(40,0){\line(1,1){40}}
\put(40,0){\line(-1,1){40}}
\put(40,80){\line(1,-1){40}}
\put(40,80){\line(-1,-1){40}}
%%%%%%%%
\put(0,0){\circle*{5}}\put(3,3){$F_{11}$}
\put(0,40){\circle*{5}}\put(-18,35){$L_{12}$}
\put(0,80){\circle*{5}}\put(-19,75){$F_{14}$}
\put(20,20){\circle*{5}}\put(23,23){$F_{12}$}
\put(20,60){\circle*{5}}\put(16,47){$F_{13}$}
\put(40,0){\circle*{5}}\put(34,5){$L_{12}'$}
\put(40,80){\circle*{5}}\put(35,85){$L_{34}'$}
\put(60,20){\circle*{5}}\put(46,23){$F_{31}$}
\put(60,60){\circle*{5}}\put(52,48){$F_{34}$}
\put(80,0){\circle*{5}}\put(83,2){$F_{32}$}
\put(80,40){\circle*{5}}\put(84,35){$L_{34}$}
\put(80,80){\circle*{5}}\put(83,76){$F_{33}$}
%%%%%%%
\end{picture}
\caption{Smooth rational curves on $S_1$}
\label{12A}
\end{figure}
On the other hand, by using Lemma \ref{splitting}, we see that
the dual graph $K_4'$ of the curves
$D_k\ (k=1,\dots,4)$ is the complete graph in four vertices
with edges doubled. There are edges connecting Figure \ref{12A} and $K_4'$ as follows:
each $D_k$ has two double edges connecting to
two vertices in Figure \ref{12A}, which can be detected from \eqref{11}.
To be explicit, we have
\begin{equation}\label{12A+K4'}
\begin{split}
(D_1,F_{11})=(D_1,F_{33})=2,\quad (D_2,F_{14})=(D_2,F_{32})=2\\
(D_3,F_{13})=(D_3,F_{31})=2,\quad (D_4,F_{12})=(D_4,F_{34})=2
\end{split}
\end{equation}
and there are no other edges between $K_4'$ and Figure \ref{12A}.
We summarize as follows.
\begin{proposition}
The Enriques surface $S_1$ has $12+4=16$ smooth rational curves on it. Their configuration
is described by Figure \ref{12A}, the graph $K_4'$ and the relation \eqref{12A+K4'}.
\end{proposition}
We can choose the rational curves on $S_1$ defined for example by
\begin{equation}\label{SingS_1}
F_{14}+F_{13}+L'_{34}+F_{34}+L_{34}+F_{31}+L'_{12}+F_{12} \ \cup \ D_1.
\end{equation}
This gives the resolution graph of rational double points of type $D_8$ and $A_1$.
By Artin's contractability theorem(\cite{A}), we obtain the following.
\begin{proposition}
There exists a birational morphism $S_1\rightarrow \overline{S}_1$ which contracts
the nine curves above \eqref{SingS_1} to rational double points of type $D_8 \oplus A_1$.
The surface $\overline{S}_1$  is a
Gorenstein $\Q$-homology projective plane.
\end{proposition}
This gives the first surface in Theorem \ref{ex-thm}.\\

\noindent {\bf{The Kondo-Mukai involution.}}\quad The second free involution $\varepsilon_2$
is constructed as follows. As is introduced in \cite{Ohashi07},
we project the quadric surface $R\subset \P^3$ onto $\P^2$ from the
point $p_{44}$.
Then the double cover $\overline{X}/R$ is birationally transformed into the second double
cover
$\overline{X'}\rightarrow \P^2=R'$.
If we use the homogeneous coordinates $(x_0:x_1:x_2)$
of $R'$ with the birational isomorphism $x_1/x_0=x, x_2/x_0=x'$ to $R$, then
the branch of $\overline{X'}/R'$ is given by the union of lines
\[\begin{split}
l_1\colon x_1=0; \quad l_2\colon x_1=x_0; \quad l_3\colon x_1=\lambda x_0;\\
l_1'\colon x_2=0; \quad l_2'\colon x_2=x_0; \quad l_3'\colon x_2=\lambda x_0.
\end{split}\]
Here and in what follows, we denote the curves and points of $R'$
corresponding to those of $R$ by the same letter. We note that the point $p_{44}\in R$
is blown up to the line at infinity of $R'$.

The Kondo-Mukai involution $\varepsilon_2$ of $X$ is induced from the Cremona involution
of $R'$ whose center is $\{p_{12}, p_{21}, p_{33}\}$ and which interchanges the two points
$(0:1:0)$ and $(0:0:1)$. The latter two points are the blow-downs of lines
$l_4$ and $l_4'$ respectively.
By the change of coordinates
\begin{equation*}
\begin{split}
y_0&=x_1+x_2-x_0,\\
y_1&=x_1-(1-\lambda^{-1})x_2-x_0,\\
y_2&=x_2-(1-\lambda^{-1})x_1-x_0,
\end{split}
\end{equation*}
the Cremona involution is given by the formula
\[(y_0:y_1:y_2)\mapsto \left( \frac{1}{y_0}: \frac{-1+\lambda^{-1}}{y_1}:\frac{-1+\lambda^{-1}}{y_2}\right).\]
Since the four isolated fixed points are not on the branch lines, we see in fact that
a suitable lift $\varepsilon_2$ is a fixed point free involution of $X$. This is one of what
is known as the {\em{Kondo-Mukai involution}}.
\begin{proposition}
The involution $\varepsilon_2$ acts on the curves on $X$ as follows. We also define
the notation of the corresponding curves on $S_2=X/\varepsilon_2$ by the table.
\begin{figure}[h]
\begin{longtable}{c|c}
On $X$ & Notation \\ \hline
$E_{11} \leftrightarrow E_{22}$ & $F_{11}$ \\
$E_{13} \leftrightarrow E_{32}$ & $F_{13}$ \\
$E_{14} \leftrightarrow E_{42}$ & $F_{14}$ \\
$E_{23} \leftrightarrow E_{31}$ & $F_{23}$ \\
$E_{24} \leftrightarrow E_{41}$ & $F_{24}$ \\
$E_{34} \leftrightarrow E_{43}$ & $F_{34}$ \\
$L_{1} \leftrightarrow L'_{2}$ & $L_{12}$ \\
$L_2 \leftrightarrow L'_1$ & $L_{21}$ \\
$L_3 \leftrightarrow L'_3$ & $L_{33}$ \\
$L_4 \leftrightarrow L'_4$ & $L_{44}$
\end{longtable}
\caption{The action of $\epsilon_2$ on the $20$ rational curves on $X$}
\end{figure}
\end{proposition}

\begin{proof}
All cases can be verified by computations and easy combinatorial observations.
\end{proof}
We also see that $\varepsilon_2$ exchanges $C_1$ and $C_1'$, which produce
the smooth rational curve $D_1$ on $S_2$.
The image $\varepsilon_2(E_{12})$ of the curve $E_{12}$ is the
inverse image of the line connecting $p_{21}$ and $p_{33}$ and they produce the
smooth rational curve $B$ on $S_2$.
The dual graph of these twelve curves introduced so far is depicted in Figure \ref{10A}.

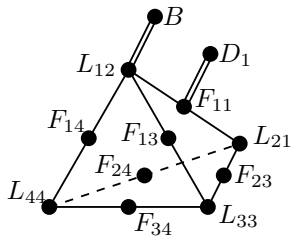
\begin{figure}[h]
\centering
\begin{picture}(80,90)
\linethickness{0.7pt}
\put(0,10){\circle*{6}}
\put(-16,13){$L_{44}$}
\put(30,10){\circle*{6}}
\put(60,10){\circle*{6}}
\put(64,4){$L_{33}$}
\put(30,62){\circle*{6}}
\put(10,61){$L_{12}$}
\put(15,36){\circle*{6}}
\put(-1,40){$F_{14}$}
\put(45,36){\circle*{6}}
\put(72,34){\circle*{6}}
\put(77,34){$L_{21}$}
\put(66,22){\circle*{6}}
\put(51,48){\circle*{6}}
\put(55,48){$F_{11}$}
\put(36,22){\circle*{6}}
\put(70,19){$F_{23}$}
\put(32,1){$F_{34}$}
\put(17,23){$F_{24}$}
\put(27,35){$F_{13}$}
%%%
\put(0,10){\line(1,0){60}}
\put(60,10){\line(-15,26){30}}
\put(30,62){\line(-15,-26){30}}
\put(30,62){\line(3,-2){40}}
\put(72,34){\line(-1,-2){12}}
\multiput(0,10)(6,2){12}{\line(3,1){3}}
%%%%%%
%\put(29,30){\line(0,-1){25}}
%\put(31,30){\line(0,-1){25}}
%\put(30,5){\circle*{6}}
%\put(34,5){$D_1$}
\put(50,48){\line(1,2){10}}
\put(52,48){\line(1,2){10}}
\put(61,68){\circle*{6}}
\put(64,65){$D_1$}
\put(29,62){\line(1,2){10}}
\put(31,62){\line(1,2){10}}
\put(40,82){\circle*{6}}
\put(43,79){$B$}
\end{picture}
\caption{Smooth rational curves on $S_2$}
\label{10A}
\end{figure}

We can choose the rational curves on $S_2$ defined by
\begin{equation}\label{SingS_2}
F_{14}+L_{44}+F_{34}+L_{33}+F_{13}+F_{23}+L_{21}\ \cup B\ \cup D_1.
\end{equation}
This gives the resolution graph of rational double points of type $E_7 \oplus 2A_1$.
As in the first example, we obtain the following second surface of theorem \ref{ex-thm}.
\begin{proposition}
There exists a birational morphism $S_2\rightarrow \overline{S}_2$ which contracts
the nine curves above \eqref{SingS_2} to rational double points of type $E_7 \oplus 2A_1$.
The surface $\overline{S}_2$ is a
Gorenstein $\Q$-homology projective plane.
\end{proposition}
This completes the proof of Theorem \ref{ex-thm}.

\smallskip

\noindent\textbf{Acknowledgments.} This work was supported by Basic Science Research Program through the National Research Foundation of Korea (NRF-2011-0022904) and by the National Research Foundation of Korea (NRF-2007-C00002). The authors are grateful to the referees for careful reading and crucial suggestions. DongSeon Hwang would like to thank Kenji Hashimoto for useful discussion. 

\bigskip
\noindent
{\bf Added in Proof.} The authors have been informed that the type $A_3 \oplus 3A_2$ realizes on the Enriques surface constructed by Slawomir Rams and Matthias Sch\"utt [On Enriques surfaces with four cusps, arXiv:1404.3924, Example 3.9]. The Enriques surface admits an elliptic fibration with four $I_3$-fibres, one of them a double fibre, such that the configuration of the 12 nodal curves and a suitable double section contains the configuration $A_3 \oplus 3A_2$.

%%%%%%%%%%%%%%%%%%%%%%%%%%%%%%%%%%%%%%%%%%%%%%%%%%%%%%%%%%%%%%%%%%%%%%%%
%\bibliographystyle{amsplain}

\end{document}